\documentclass[12pt]{article}
\usepackage[T2A]{fontenc}
\usepackage{graphicx,xcolor}
\usepackage{amssymb,amsmath,amsfonts,amsthm,amscd,latexsym,verbatim,graphics,epsfig,indentfirst,color}
\usepackage{geometry}
\newtheorem{theorem}{Theorem}[section]   
\newtheorem{corollary}[theorem]{Corollary}
\newtheorem{lemma}[theorem]{Lemma}
\newtheorem{proposition}[theorem]{Proposition}

\newtheorem{example}[theorem]{Example}

\begin{document}

\begin{center}
{\Large
Radicals of Lie-Solvable Novikov algebras}

Alexander S. Panasenko
\end{center}

\begin{abstract}
We show that, for a Lie-solvable Novikov algebra, the Baer radical is precisely the set of all right nilpotent elements, whereas the Andrunakievich radical is the largest left quasiregular ideal. We investigate the extent to which certain properties of associative and commutative algebras equipped with a derivation are preserved under the Gelfand–Dorfman construction.


\end{abstract}

\section{Introduction}

An algebra $A$ is called a {\bf Novikov algebra} if it satisfies the identities:
\[(x,y,z)=(y,x,z), \qquad (xy)z=(xz)y,\]
where $(x,y,z)=(xy)z-x(yz)$ is the associator.

The first appearance of algebras satisfying these identities dates back to 1979, in the work of I.~M.~Gel'fand and I.~Ya.~Dorfman \cite{GD1979}. Later, right-symmetric and left-commutative algebras emerged in the works of A.~A.~Balinskii and S.~P.~Novikov~\cite{BN1985,N1985}. The first systematic algebraic study of Novikov algebras was carried out by E.~I.~Zelmanov \cite{Zelm1987}, who proved that every simple finite-dimensional Novikov algebra over an algebraically closed field of characteristic~0 is isomorphic to the ground field. As a consequence of the results obtained by V.~Dotsenko, N.~Ismailov, and U.~Umirbaev in \cite{DIU}, every simple finite-dimensional Novikov algebra over a field of characteristic~0 is commutative. V.~T.~Filippov constructed the first examples of noncommutative simple Novikov algebras in \cite{Phil1989}. In \cite{Xu1996}, X.~Xu classified all simple finite-dimensional Novikov algebras over an algebraically closed field of characteristic $p>0$. A complete classification of simple finite-dimensional Novikov algebras over arbitrary fields of characteristic $p>0$ was obtained by V.~N.~Zhelyabin and A.~P.~Pozhidaev in \cite{PZ2026}; see also \cite{ZZ2024}. In \cite{Xu2001}, X.~Xu obtained a description of infinite-dimensional simple Novikov algebras over an algebraically closed field of characteristic zero under certain additional conditions, thereby correcting several errors in \cite{Osborn1994}.

Since the beginning of the 21st century, Novikov algebras satisfying polynomial identities have been studied extensively. V.~T.~Filippov established in \cite{Phil2001} that every left nil Novikov algebra of bounded index over a field of characteristic~0 is nilpotent. Later, A.~S.~Dzhumadildaev and K.~M.~Tulenbaev \cite{DT2006} showed that every right nil Novikov algebra of bounded index $n$ over a field of characteristic~0, or of characteristic greater than $n$, is right nilpotent.

A fundamental result obtained by I.~P.~Shestakov and Z.~Zhang in \cite{ShestZhang} states that, for a Novikov algebra $A$, the following conditions are equivalent:
\begin{itemize}
\item $A$ is solvable;
\item $A$ is right nilpotent;
\item $A^2$ is nilpotent.
\end{itemize}

In \cite{TUZ}, V.~N.~Zhelyabin, K.~M.~Tulenbaev, and U.~U.~Umirbaev established that the commutator ideal of any Lie-solvable Novikov algebra over a field of characteristic different from 2 is right nilpotent. They also established that the commutator ideal of every Lie-metabelian Novikov algebra is nilpotent of index 4. In \cite{DIU}, V.~Dotsenko, N.~Ismailov, and U.~Umirbaev proved that every Novikov algebra satisfying a polynomial identity over a field of characteristic~0 is Lie-solvable and right-associator-nilpotent, and has a right nilpotent commutator ideal. In \cite{ZP2024}, V.~N.~Zhelyabin and A.~S.~Panasenko proved that, for a Novikov algebra over a field of characteristic different from 2, the commutator ideal is solvable if and only if the algebra is right-associator-nilpotent. In \cite{UZ2021}, V.~N.~Zhelyabin and U.~U.~Umirbaev studied the solvability of graded Novikov algebras.

The paper is organized as follows. Section 2 is devoted to the study of right nil Novikov algebras. In particular, we prove that an extension of a right nilpotent ideal by a right nil algebra is again a right nil algebra. Section 3 is devoted to the study of radicals in Lie-solvable Novikov algebras. We prove that the locally solvable radical and the right nil radical of a Lie-solvable Novikov algebra exist and coincide with the Baer radical. We prove that every Lie-solvable Novikov algebra has a left quasiregular radical, which coincides with the Andrunakievich radical. In Section 4, we study the stability of several properties under the Gelfand--Dorfman construction and present a number of examples.

\section{Nil Novikov Algebras}

In any Novikov algebra $A$, the following identity holds:
\begin{equation}
    (x,y,z)t = (xt,y,z) = (x,yt,z)\label{(x,y,z)t}
\end{equation}
for all $x,y,z,t\in A$.

We define the {\bf left-normed powers} of an element $x\in A$ recursively as follows.

\begin{itemize}
\item $x^1 = x$;
\item For $n>1$, we define $x^n = x^{n-1}x$.
\end{itemize}

An element $x\in A$ is called \textbf{r-nilpotent} if $x^n = 0$ for some $n>0$. A Novikov algebra $A$ is called {\bf r-nil} if every element of $A$ is r-nilpotent. A one-sided ideal $I$ of $A$ is called {\bf r-nil} if every element of $I$ is r-nilpotent.

\medskip We denote by $R_x$ the operator of right multiplication by $x\in A$.

\begin{lemma}\label{lem1}
    Let $A$ be a Novikov algebra and let $x\in A$ be such that 
    \[(x^n)^2 = (x^{n+1})^2 = 0\] 
    for some $n>0$. Then $x^{2n+2} = 0$.
\end{lemma}

\begin{proof} First, note that for every $m=1,\dots,n$, identity $\eqref{(x,y,z)t}$ yields
\[(x^n,x^m,x)R_x^{n-m+1} = (x^n,x^n,x)x = (x^n)^2R_x^2 - (x^n\cdot x^{n+1})x = (x^n)^2R_x^2 - (x^{n+1})^2 = 0.\]
The last equality follows from right commutativity. Thus,
\begin{equation}
(x^n,x^m,x)R_x^{n-m+1} = 0. \label{l1}
\end{equation}
We prove by induction on $k$ that 
\[x^{2n+2} = (x^n\cdot x^k)R_x^{n+2-k}\] 
for any $k=1\dots, n$. The base case $k=1$ follows directly from the definition of the left-normed powers. Assume that the statement holds for $k=m$. Then
\begin{align*}x^{2n+2} = (x^n\cdot x^m)R_x^{n+2-m}&\\ = (x^n,x^m,x)R_x^{n+2-m-1} + (x^n\cdot x^{m+1}) R_x^{n+2-m-1}& = (x^n\cdot x^{m+1})R_x^{n-m+1},
\end{align*}
where the last equality follows from the identity $\eqref{l1}$. Applying the above identity with $k=n$, we obtain
\[x^{2n+2} = (x^n\cdot x^n)R_x^2 = 0.\]
\end{proof}

\begin{lemma}\label{lem2}
    Let $I$ be an r-nil right ideal of index~2 in a Novikov algebra $A$. If $x\in A$ and $x^n\in I$ for some $n>0$, then $x^{2n+2} = 0$.
\end{lemma}

\begin{proof} 
    Indeed, if $x^n\in I$, then $x^{n+1}=x^n\cdot x\in I$ and $(x^n)^2=(x^{n+1})^2=0$. Hence, by Lemma~\ref{lem1}, $x^{2n+2}=0$.
\end{proof}

\begin{lemma}\label{lem3}
    Let $I$ be an ideal of a Novikov algebra $A$. If $x^n\in I$ for some $x\in A$ and some $n>0$, then $x^{2n+2}\in I^2$.
\end{lemma}

\begin{proof}
 Indeed, consider the quotient algebra $A/I^2$. It contains the ideal $I/I^2$, whose multiplication is trivial. In particular, $I/I^2$ is either the zero ideal or an r-nil ideal of index~2. In either cases, Lemma~\ref{lem2} implies that $x^{2n+2}+I^2 = 0+I^2$, and hence $x^{2n+2}\in I^2$.
\end{proof}

We define the {\bf left-normed powers} of $A$ recursively as follows.
\begin{itemize}
\item $A^{[1]}=A$;
\item $A^{[k]}=A^{[k-1]}A$.
\end{itemize}

An algebra $A$ is called \textbf{right nilpotent} if $A^{[k]}=0$ for some positive integer $k$.

\begin{theorem}\label{theo1} 
Let $I$ be an ideal of a Novikov algebra $A$. If $x^n\in I$ for some $x\in A$ and some $n>0$, then 
\[x^{s_k}\in I^{[k]}\] 
for every $k\in\mathbb{N}$, where 
\[s_1 = n, \qquad s_k=2s_{k-1}+2.\] 
In particular, if $I$ is a right nilpotent ideal and $A/I$ is an r-nil algebra, then $A$ is an r-nil algebra.
\end{theorem}

\begin{proof} 
    We prove the statement by induction on $k$. For $k=2$, the statement follows from Lemma~\ref{lem3}. Assume that $x^{s_k}\in I^{[k]}$. Then, by Lemma~\ref{lem3}, $x^{2s_k+2}\in (I^{[k]})^2$. Since $(I^{[k]})^2\subseteq I^{[k]}I = I^{[k+1]}$, it follows that $x^{2s_k+2}\in I^{[k+1]}$. 
\end{proof}

The question of whether every extension of an r-nil algebra by an r-nil algebra is again r-nil remains open in the class of Novikov algebras. An affirmative answer to this question would imply the existence of a nilradical in the class of Novikov algebras.

\section{Lie-solvable Novikov Algebras}

Throughout this section, we assume that $F$ is a field of characteristic different from~$2$.

Let $I$ be an ideal of a Novikov algebra $A$. Then the subspace $[I,I]$ spanned by all commutators of elements of $I$ is an ideal of $A$. We define the {\bf derived series} of $A$ recursively by
\[A^{(0)_L}=A, \quad A^{(n+1)_L}=[A^{(n)_L},A^{(n)_L}].\]
A Novikov algebra is called {\bf Lie-solvable} if $A^{(n)_L}=0$ for some $n\ge 0$.

The class of Lie-solvable Novikov algebras is rather extensive. In characteristic~0, it coincides with all PI-algebras and in particular, contains all finite-dimensional algebras~\cite{DIU}.

An ideal $I$ of an algebra is called {\bf trivial} if its multiplication is identically zero, that is $I^2 = 0$.

\medskip Let us recall the definition of the Baer radical. Set $\mathcal{B}_0(A)=0$, and define $\mathcal{B}_1(A)$ to be the sum of all trivial ideals of an algebra $A$. Next, we define $\mathcal{B}_\alpha(A)$ recursively. If $\alpha$ is a limit ordinal, then \[\mathcal{B}_{\alpha}(A)=\cup_{\beta<\alpha}\mathcal{B}_{\beta}(A).\] 
If $\alpha$ is a successor ordianl, then $\mathcal{B}_{\alpha}(A)$ is defined as the full preimage in $A$ of the ideal $\mathcal{B}_1(A/\mathcal{B}_{\alpha - 1}(A))$ under the natural quotient map $A\to A/\mathcal{B}_{\alpha - 1}(A)$. In other words, $$\mathcal{B}_\alpha(A)/\mathcal{B}_{\alpha - 1}(A)=\mathcal{B}_1(A/\mathcal{B}_{\alpha - 1}(A)).$$ We denote by $\mathcal{B}(A)$ the term $\mathcal{B}_{\delta}(A)$, where $\delta$ is an ordinal satisfying $\delta>|A|$. The ideal $\mathcal{B}(A)$ is called the \textit{\textbf{Baer radical}} of $A$.

In \cite{P2022}, it was shown that the class of Novikov algebras admits a Baer radical; that is, the class of algebras coinciding with their Baer radical forms a radical class within the variety of Novikov algebras.

\medskip In \cite{TUZ}, it was proved that the commutator ideal $[A,A]$ of a Lie-solvable Novikov algebra over a field of characteristic $\neq 2$ is right nilpotent. In \cite{ShestZhang}, it was shown that the notions of solvability and right nilpotency coincide in Novikov algebras. We will freely use these facts in what follows.

\begin{lemma}\label{lem4} 
Let $A$ be a Lie-solvable Novikov algebra. Then the Baer radical of $A$ is given by
\[\mathcal{B}(A)=\{x\in A\mid x^n = 0 \text{\textup{ for some }} n>0\}.\]
\end{lemma}

\begin{proof} Let $x\in\mathcal{B}(A)$. Since $\mathcal{B}(A)$ contains all solvable ideals and every right nilpotent ideal is solvable, it foolows that $[A,A]\subseteq \mathcal{B}(A)$. The ideal $\mathcal{B}(A)/[A,A]$ is the Baer radical ideal of the quotient algebra $A/[A,A]$. Since the Baer radical of an associative commutative algebra coincides with its nilradical, it follows that $\mathcal{B}(A)/[A,A]$ is an r-nil ideal. Since the ideal $[A,A]$ is right nilpotent, it follows by Theorem~\ref{theo1} that $\mathcal{B}(A)$ is an r-nil ideal.

Let $x\in A$ satisfy $x^n = 0$. Let $\overline{x}$ denote the image of $x$ in the quotient algebra $A/[A,A]$. Since $\overline{x}$ is nilpotent, we have $\overline{x}\in\mathcal{B}(A/[A,A])$, because the Baer radical of an associative commutative algebra coincides with its nilradical. Since Baer radical exists in the class of Novikov algebras and the ideal $[A,A]$ is solvable, it follows that the preimage of $\mathcal{B}(A/[A,A])$ is contatined in $\mathcal{B}(A)$. Hence, $x\in \mathcal{B}(A)$. 
\end{proof}

In the class of associative commutative algebras, the Baer radical coincides with both the locally nilpotent radical and the nilradical. We obtain an analogue of this result for Lie-solvable Novikov algebras (note that nilpotency and solvability do not coincide even in the finite-dimensional case).

\begin{theorem}\label{theo2}
    In the class of Lie-solvable Novikov algebras, the following classes coincide:
    \begin{enumerate}
        \item The class of Baer-radical algebras;
        \item The class of r-nil algebras;
        \item The class of locally solvable algebras.
    \end{enumerate}
\end{theorem}

\begin{proof} The classes of r-nil algebras and Baer-radical Lie-solvable Novikov algebras coincide by Lemma~\ref{lem4}.

Let $A$ be a locally solvable Novikov algebra. Then it is locally right nilpotent, and hence an r-nil algebra. Therefore, the class of locally solvable Novikov algebras is contained in the class of r-nil Novikov algebras.

Let $A$ be a Baer-radical Lie-solvable Novikov algebra. Let $x_1,\dots,x_n$ be arbitrary elements of $A$. Since $A/[A,A]$ is an associative commutative Baer-radical algebra, it is locally solvable; hence, the subalgebra $B$ generated by the elements $x_1,\dots,x_n$ satisfies $B^{(m)}\subseteq [A,A]$ for some $m>0$. Since the ideal $[A,A]$ is solvable, it follows that $B^{(m+s)}=(B^{(m)})^{(s)}=0$ for some $s>0$. Hence, $A$ is locally solvable. 
\end{proof}

\medskip To avoid any ambiguity, we begin by recalling the following definition. An associative and commutative algebra is called an integral domain if it contains no nonzero zero divisors.

\begin{corollary}\label{cor1} Let $A$ be a Lie-solvable Novikov algebra. Then the Baer radical $\mathcal{B}(A)$ is precisely the intersection of all ideals $I\subseteq A$ such that $A/I$ is an integral domain.
\end{corollary}

\begin{proof} It is well known that the Baer radical is the intersection of all ideals $I$ of $A$ such that $A/I$ is prime. However, in a Lie-solvable semiprime Novikov algebra $A$, the commutator square $[A,A]$ vanishes. This means that every semiprime Novikov algebra is commutative. It follows that every prime quotient $A/I$ is an integral domain.
\end{proof}

\begin{corollary}\label{cor2}
    Every semiprime Lie-solvable Novikov algebra is a subdirect product of integral domains.
\end{corollary}

\textbf{Definition.} An element $x\in A$ is called \textbf{left quasiregular} if there exists an element $y\in A$ such that $x+y=yx$. An algebra is called  \textbf{left quasiregular} if every element of the algebra is left quasiregular.

Let $\mathcal{C}$ denote the class of all subdirectly irreducible Novikov algebras with an idempotent heart. An ideal $I$ of Novikov algebra $A$ is called a $\mathcal{C}$-ideal if the quotient algebra $A/I$ belongs to the class $\mathcal{C}$. The class $\mathcal{A}$ consisting of all Novikov algebras that do not admit homomorphic images in the class $\mathcal{C}$ is a radical class (proved in \cite{Zhevl}). The largest ideal of an algebra $A$ belonging to the class $\mathcal{A}$ is denoted by $\mathcal{A}(A)$ and is called the \textbf{Andrunakievich radical} of the algebra $A$. The Andrunakievich radical exists in the variety of all non-associative algebras.

\begin{lemma}\label{lem5}
Let $A$ be a Lie-solvable Novikov algebra, $I$ an ideal of $A$, and suppose that both $I$ and $A/I$ are left quasiregular. Then $A$ is left quasiregular as well. Moreover, $A=\mathcal{A}(A)$.
\end{lemma}

\begin{proof} Let $\overline{I}$ denote the image of $I$ in the quotient algebra $A/[A,A]$. It follows that $A/\overline{I}$ is left quasiregular, since it is a homomorphic image of $A/I$. On the other hand, $\overline{I}$ is left quasiregular, since it is an image of the left quasiregular algebra $I$. But the quotient algebra $A/[A,A]$ is associative and commutative, and left quasiregularity is preserved under extensions in the variety of associative commutative algebras. Therefore, the algebra $A/[A,A]$ is left quasiregular. Since the Jacobson radical coincides with the Andrunakievich radical in the variety of associative commutative algebras, it follows that $\mathcal{A}(A/[A,A]) = A/[A,A]$. Moreover, every solvable algebra is $\mathcal{A}$-radical, since no solvable algebra can be homomorphically mapped onto a semiprime algebra, in particular onto a subdirectly irreducible algebra with an idempotent heart. Therefore, $\mathcal{A}([A,A])=[A,A]$. Since the Andrunakievich radical exists, it follows that $\mathcal{A}(A)=A$.

We will prove by induction on $n$ that, for every $x\in A$ and every $n>0$, the element $x+[A,A]^{[n]}$ is left quasiregular in the quotient algebra $A/[A,A]^{[n]}$. The case $n=1$ follows from the fact established above that the algebra $A/[A,A]$ is left quasiregular. We will show that $x+[A,A]^{[n+1]}$ is left quasiregular in the algebra $A/[A,A]^{[n+1]}$. By the induction hypothesis, there exists an element $y\in A$ such that $x+y-yx = -v\in [A,A]^{[n]}$. Note that $v^2 = v[y,x] = (v,v,y) = 0$. Then
\begin{align*}
x+(y+v-vy+(v,y,y))-(y+v-vy+(v,y,y))x & \\ =x+y+v-vy-yx-vx+(vy)x + (v,y,y)-(v,yx,y) & \\  =-vy-vx+(vy)x +(v,y-yx,y) & \\ =v^2-v(yx) +(vy)x + (v,-v-x,y) & \\ =v^2-v[y,x]-v(xy)+(vx)y - (v,x,y) - (v,v,y) &= 0.
\end{align*}
Thus, for every $n>0$ the element $x+[A,A]^{[n]}$ is left quasiregular in the quotient algebra $A/[A,A]^{[n]}$. By the results of \cite{TUZ}, the ideal $[A,A]$ is right nilpotent. Thus, there exists a positive integer $n$ such that $[A,A]^{[n]}=0$. Hence, $x$ is a left quasiregular element in $A$.
\end{proof}

\begin{theorem}\label{theo3} 
    In the class of Lie-solvable Novikov algebras, left quasiregular algebras coincide with $\mathcal{A}$-radical algebras.
\end{theorem}    

\begin{proof} Let $A$ be a Lie-solvable Novikov algebra. We define $J(A)$ to be the sum of all left quasiregular idealss of $A$. Let $I$ and $K$ be left quasiregular ideals. Then $(I+K)/K\simeq I/(I\cap K)$, and hence $I+K$ is left quasiregular by Lemma~\ref{lem5}. By induction, any finite sum of left quasiregular ideals is left quasiregular. Consequently, the sum of all left quasiregular ideals is left quasiregular, since each of its elements is contained in a finite sum of left quasiregular ideals.

If the algebra $A/J(A)$ contains a left quasiregular ideal $\overline{L}$, then its preimage $L$ in $A$ is left quasiregular by Lemma~\ref{lem5}. However, it follows that $L\subseteq J(A)$, and hence $\overline{L}=\overline{0}$.

Lemma~\ref{lem5} established that every left quasiregular Lie-solvable Novikov algebra is $\mathcal{A}$-radical. Conversely, suppose that $A=\mathcal{A}(A)$. Then the algebra $A/[A,A]$ is $\mathcal{A}$-radical as well, and therefore is quasiregular. On the other hand, the ideal $[A,A]$ is right nilpotent by \cite{TUZ}; hence, by \cite{P2024}, it is left quasiregular. Therefore, by Lemma~\ref{lem5}, the algebra $A$ is left quasiregular. 
\end{proof}

\begin{corollary}\label{cor3}
    Let $A$ be a Lie-solvable Novikov algebra. Then the left quasiregular radical of $A$ coincides with the intersection of all ideals $I$ such that $A/I$ is a field.
\end{corollary} 

\begin{proof} 
    By Theorem~\ref{theo3}, the left quasiregular radical coincides with the Andrunakievich radical. By \cite{P2024}, the Andrunakievich radical of $A$ coincides with the intersection of all ideals $I$ such that $A/I$ is a subdirectly irreducible algebra with an idempotent heart by \cite{P2024}. However, since the ideal $[A,A]$ is solvable, it follows that $[A,A]\subseteq \mathcal{A}(A)\subseteq I$. Therefore, $A/I$ is an associative commutative algebra. It is well known that a subdirectly irreducible associative commutative algebra with an idempotent heart is a field.
\end{proof}

\section{Stability of Radicals under the Gelfand-Dorfman Structure}

Let $B$ be an associative commutative algebra equipped with a derivation $d\in\mathrm{Der}(B)$. Then the algebra $\mathrm{GD}(B,d)$ on the space $B$ with the multiplication defined by
\[x\circ y = xd(y),\]
is a Novikov algebra and is called the {\bf Gelfand-Dorfman construction} associated with the algebra $B$ and the derivation $d$.

\begin{proposition}\label{prop1}
Let $B$ be an associative commutative algebra with a derivation $d$. Then the following statements hold for the Novikov algebra $A=\mathrm{GD}(B,d)$:
\begin{itemize}
\item If $B$ is nil, then $A$ is r-nil;
\item If $B$ is nilpotent, then $A$ is nilpotent; 
\item If $B$ is quasiregular, then $A$ is left quasiregular.
\end{itemize}
\end{proposition}

\begin{proof} Let $B$ be a quasiregular algebra and let $x\in A$. Then there exists $z\in A$ such that $d(x)+z-d(x)z=0$. Hence,
\[x+(xz-x)-(xz-x)d(x) = xz-xzd(x)+xd(x) = 0.\]
Therefore, $x\in A$ is left quasiregular with quasi-inverse $xz-x$.

Let $B$ be a nil algebra. We denote by $x^{m,\circ}$ the $m$-th power of $x$ in the algebra $\mathrm{GD}(B,d)$. If $x\in A$ and $d(x)^n = 0$, it follows that
\[x^{n+1,\circ} = xR_{d(x)}^{n} = xd(x)^n = 0.\]
Let $B$ be a nilpotent algebra such that $B^n = 0$, and let $x_1,\dots,x_n\in A$. Then
\[y=x_1\circ(x_2\circ\dots \circ x_n\dots) = x_1d(x_2d(x_3\dots d(x_n)\dots ).\]
It remains to note that $x_{k+1}d(x_{k+2}\dots d(x_n))\in B^{n-k}$. Then $y\in B^n=0$. Therefore, the algebra $A$ is left nilpotent. It follows from \cite{ShestZhang} that $A$ is nilpotent. 
\end{proof}

\begin{example} Let $F$ be a field of characteristic~0, let $C = F[x_1,x_2,\dots]$ be a free associative commutative algebra in countably many variables, and let $B=C/(x_1^2,x_2^2,\dots)$. Since $B$ is $\mathbb{N}$-graded, it admits a derivation $d$ acting on monomials $f$ by \[d(f)=\mathrm{deg}(f)f.\] The algebra $B$ is a nil algebra, but not a nilpotent algebra.

It follows from Proposition~\ref{prop1} that the algebra $A=\mathrm{GD}(B,d)$ is r-nil. However, the algebra $A$ is neither nilpotent nor right nilpotent. Therefore, the Gelfand-Dorfman construction of a nil algebra is not necessarily right nilpotent.
\end{example}

\textbf{Definition}. An element $x\in A$ is called \textbf{right quasiregular} if there exists $y\in A$ such that $x+y=xy$. An algebra is called \textbf{right quasiregular} if every element in it is right quasiregular..

\begin{example} Let $F$ be a field of characteristic $0$. Let $B$ be the algebra of rational functions $f/g$, where $f(x),g(x)\in F[x]$, $f(0)=0$ and $g(0)\neq 0$. Then $B$ is a quasiregular algebra. Consider the derivation $d$ on $B$ defined by the rule
\[d(f/g) = x(f'g-fg')/g^2,\]
where $f'$ denotes the usual derivative of the polynomial $f$. Let $A=\mathrm{GD}(B,d)$. By Proposition~\ref{prop1}, $A$ is left quasiregular. We show that the element $x\in A$ is not right quasiregular. Let
\[x + f/g - x^2(f'g-fg')/g^2 = 0.\]
Then
\[r(x) = xg^2 + fg - x^2(f'g-fg') = 0.\]
Let $\deg(g) = m$ and $\deg (f) = n$, where $g(x) = \sum\limits_{k=0}^m \beta_k x^k$, $\beta_m\neq 0$, and $f(x) = \sum\limits_{k=1}^n \alpha_kx^k$, $\alpha_n\neq 0$. Note that in $r(x)$, the leading monomial of the first term is $\beta_m^2 x^{2m+1}$, that of the second term is $\alpha_n\beta_m x^{n+m}$, and that of the last term is $(n-m)\alpha_n\beta_m x^{n+m+1}$, assuming $n\neq m$. Thus, if $m>n$, then the leading monomial is $\beta_m^2x^{2m+1}$, which forces $\beta_m = 0$, a contradiction. If $n>m$, then the leading monomial is $(n-m)\alpha_n\beta_m x^{n+m+1}$, which forces $\alpha_n\beta_m=0$, a contradiction. Hence, $n=m$. But in this case, the leading monomial again has the form $\beta_m^2 x^{2m+1}$, which implies $\beta_m = 0$, a contradiction. In any case, we obtain a contradiction; hence the element $x$ is not right quasiregular. Hence, $d(x)=x$ and $x^{k,\circ}=x^k\neq 0$ for all $k\ge 1$. It follows that $x$ is not r-nilpotent.
\end{example}

Therefore, the Gelfand–Dorfman construction of a quasiregular algebra is not necessarily right quasiregular or r-nil.

\section*{Acknowledgments}

The author is grateful to V. N. Zhelyabin for his remarks, which helped improve this paper.

The author is grateful to the anonymous referee for useful remarks.

This research is supported by the Russian Science Foundation (project 25-41-00005).

\noindent Alexander Panasenko \\
Sobolev Institute of Mathematics \\
Acad. Koptyug ave. 4, 630090 Novosibirsk, Russia \\
e-mail: a.panasenko@g.nsu.ru

\end{document}